\documentclass[12pt]{amsart}
\usepackage{amsmath}
\usepackage{amssymb}
\usepackage{amsfonts}
\usepackage{amsthm}
\usepackage{verbatim}
\usepackage{amscd}
\usepackage{cite}
\usepackage{leftidx}
\usepackage{enumerate}
\usepackage{txfonts}
\usepackage{manfnt}
\usepackage{amscd}
\usepackage[mathscr]{eucal}
\usepackage{hyperref}
\usepackage{mathpazo}
\def\inv{^{-1}}

\def\refp #1.{(\ref{#1})}

\def\sbr #1.{^{[#1]}}
\def\sfl #1.{^{\lfloor #1\rfloor}}

\def\inv{^{-1}}
\def\?{{\bf{??}}}

\def\a{{\frak a}}

\def\C{\mathbb C}
\def\P{\mathbb P}

\def\Z{\mathbb Z}

\def\ord{\text{\rm ord}}

\def\Q{\mathbb Q}

\def\O{\mathcal O}

\def\g{\mathfrak g}

\def\1/2{\frac{1}{2}}

\def\simto{\stackrel{\sim}{\rightarrow}}
\def\2{{[2]}}

\def\nl{\newline}

\def\<{\langle}
\def\>{\rangle}

\def\2{{[2]}}

\def\scl #1.{^{\lceil#1\rceil}}
\def\spr #1.{^{(#1)}}
\def\sbc #1.{^{\{#1\}}}
\def\supp{\text{supp}}
\def\subpr#1.{_{(#1)}}

\def\beq{\begin{equation*}}
\def\eeq{\end{equation*}}

\def\g3{{\Gamma\spr 3.}}

\newcommand{\eqspl}[2]{
\begin{equation}\label{#1}
\begin{split}
#2\end{split}\end{equation}}
\newcommand{\eqsp}[1]{\begin{equation*}
\begin{split}#1\end{split}\end{equation*}}

\newcommand{\exseq}[3]{
0\to #1\to #2\to #3\to 0
}

\newcommand{\beginalphaenum}{
\begin{enumerate}\renewcommand{\labelenumi}{ }
\item \begin{enumerate}
}

\def\eex{\end{rm}\end{example}}

\def\cC{\mathcal C}
\def\hH{\mathcal H}

\newtheorem{thm}{Theorem} 

\newtheorem*{thm*}{Theorem}
\newtheorem*{prop*}{Proposition}

\newtheorem*{cor*}{Corollary}

\newtheorem*{lem*}{Lemma}

\newtheorem*{claim*}{Claim}
\newtheorem{prop}[thm]{Proposition}

\theoremstyle{remark}
\newtheorem*{n&c}{Notations and conventions}

\newtheorem{rem}[thm]{Remark}
\newtheorem*{rem*}{Remark}
\newtheorem{crit-rem}[thm]{Critical remark}

\newtheorem{example}[thm]{Example}
\newtheorem*{example*}{Example}

\newtheorem*{defn*}{Definition}

\begin{document} 

\title{Incident rational curves}
\author 
{Ziv Ran}


\thanks{arxiv.org}
\date {\today}


\address {\nl UC Math Dept. \nl
Skye Surge Facility, Aberdeen-Inverness Road
\nl
Riverside CA 92521 US\nl 
ziv.ran @  ucr.edu\nl
\url{http://math.ucr.edu/~ziv/}
}

 \subjclass[2010]{14n05}

\begin{abstract}
	We study families of rational curves on an algebraic variety
	satisfying incidence conditions.
We prove an analogue of bend-and-break: that is, 
we show that under suitable conditions, such a family
must contain reducibles. In the case of curves in $\P^n$
incident to certain complete intersections, we prove the
family is irreducible.

\end{abstract}
\maketitle
Since the seminal work  of Mori and Miyaoka
    \cite{mori-pos-tang} and \cite{miyaoka-mori}, rational
    curves on algebraic varieties, especially Fano 
    manifolds, have been much studied. 
 In particular Harris and his school (see for instance
 \cite{harris-roth-starr}, \cite{riedl-yang}, \cite{tseng-note}
 and references therein) have studied the case of rational curves on
 general Fano hypersurfaces, with particular attention to the question
 of dimension and irreducibility of the family of curves of given degree.
 \par
 Our interest here is in families of rational curves on a given variety $X$
 that are incident to a fixed subvariety $Y$, i.e. meet $Y$ in a 
 specified number
 of (unspecified) points.
 This on the one hand generalizes bend-and-break, which is the case where
 $Y$ consists of 2 points, and on the other hand is related to rational
 curves on hypersurfaces, thanks to the fact (see \cite{hypersurf}) that a hypersurface
 $X_d$  of degree
 $d$ in $\P^n$ admits a 'nice' degeneration 
 (with double points only and smooth total space)
  to the union of a hypersurface 
 of degree $d-1$ with the blowup of $\P^{n-1}$ in a complete intersection
 subvariety $Y$ of type $(d-1, d)$, 
 and rational curves on $X_d$ are thereby related to rational 
 curves in $\P^{n-1}$ meeting $Y$ in a specified number of points.
\par

Here in \S 1 we present two kinds of results of bend-and-break type
(arbitrary ambient space). In the first,
we make
some disjointness conditions on the incident subvarieties, for example
(see Theorem \ref{qdisjoint}) a pair
of disjoint subvarieties $Y_1, Y_2$ meeting the curves in question in 1 
(resp.~2) points. In the second  result
(see Theorem \ref{overfilling}) we assume given  an 'overfilling' family, 
i.e. one
having at least $\infty^1$ members through a point of the ambient space,
together with a subvariety, possibly reducible, meeting the curves in 2 
points.\par
In \S 2 we specialize to the case of curves
of given degree $e$ in $\P^n$, $n\geq 3$, that are  $a$ times incident to
a  fixed
general complete intersection of type $(c, d)$ with $a\leq e$
and $c+d\leq n$. We prove in this case that the family is irreducible
and its general member is well behaved.
\vskip1cm\noindent
In this paper we will work over $\C$ (though the results probably hold over an
arbitrary algebraically closed field, at least if resolution of singularities
is known through dimension $k$).\par
We thank the referee for their helpful comments.

\par
\section{Incidental bend-and-break}
\begin{n&c}The following set-up will be in effect throughout this section.\par
\begin{enumerate}\item
$X$ is an irreducible projective variety of dimension $n\geq 3$;\item
$\pi: \mathcal C\to B$ is a proper flat family over an irreducible 
projecive base  variety of dimension
$ k\geq n-1$, with  fibres $C_b=\pi\inv(b)$,
 so that for general $b$, 
$C_b$ is a nonsingular rational curve; \item
$f:\mathcal C\to X$ is a surjective  morphism  that has degree 1 on a general fibre of $\pi$.\par
\bigskip
A family as in (iii), i.e. such that $f$ is surjective, is said to be \emph{filling}. If in addition
$\dim(B)\geq n$, so that through a point $x\in X$ there are at least $\infty^1$ curves
$f(C_b)$, it is said to be \emph{overfilling}.
\end{enumerate}


\end{n&c}
\begin{thm}\label{qdisjoint}
Under notations and conventions  as above, assume moreover there
are subvarieties $Y_1, Y_2\subset X$ of respective codimension
at least 1 (resp. at least 2) with $Y_1\cap Y_2=\emptyset$,  
such that for general $b\in B$,  $f(C_b)$  meets $Y_1$ (resp. $Y_2$) in at least
1 (resp. at least 2) points. \par
Then the family $\cC/B$ has a reducible fibre $C_b$.
\end{thm}
\begin{proof}
Assume for contradiction all fibres $C_b$ are irreducible.
With no loss of generality we may assume $\dim(B)=n-1$. 
After suitable base-change  we may assume $B$ is smooth.
Actually the argument below will use only a general curve-section of $B$,
so it's enough to assume $B$ normal.
Let $L$ be a very ample
line bundle on $B$ and set $H=f^*(\O_X(1))$ where $\O_X(1)$ is a very ample line bundle on
$X$.
Then I claim that after a further base-change we
may assume that 
 \[\mathcal C=\P(E)\] where $E$ is a rank-2 vector bundle on $B$:
 indeed, if  base-change enough so that $\pi$
 admits a section $D\subset \cC$, we can take $E=\pi_*(\O(D))$.
 Subsequently, after a further base-change, we may assume
 that $\wedge^2E$ is divisible by 2 in the Picard group,
 hence,  after a suitable twist, we may assume $\wedge^2E=\O_B$\
 and in particular, as divisors,
\eqspl{c1=0}{c_1(E)\equiv_{\mathrm{num}}0.}
We will henceforth take $c_1$ to have values in the Neron-Severi group,
so $c_1(E)=0$.
By assumption, we have  'multisections' $S'_i\subset\cC$, 
i.e. possibly reducible
subvarieties $S'_i\subset\cC, i=1,2$, generically finite  of degree 
at least 1 (resp.~at least 2) over $B$ such that
\[f(S'_i)\subset Y_i, i=1,2.\]
Base-changing via $S'_i\to B$, the pullback of $S'_i$ splits of a section.
Then after a further base-change,  we may assume we have 
3 distinct \emph{sections}
$S_1, S_2, S_3$ such that

\[f(S_1)\subset Y_1, f(S_2), f(S_3)\subset Y_2.\]
Each section $S_i$ corresponds to an exact sequence
\[\exseq{P_i}{E}{Q_i},\]
where $c_1(P_i)=-c_1(Q_i)$ thanks to $c_1(E)=0$.
Since $Y_1\cap Y_2=\emptyset$, it follows that
\eqspl{disjoint}{S_1\cap S_2=S_1\cap S_3=\emptyset}
and hence

\eqspl{1=2}{
c_1(P_1)= c_1(Q_2)=c_1(Q_3).
}
For $i=2, 3$ set $Z_i=f(S_i), m_i=\dim(Z_i)\leq n-2$.
 Note that each of  $S_2, S_3$
collapses under $f$, i.e. while $S_i$ has codimension 1 in $\mathcal C$,
 $Z_i$ has codimension 2 or more
in $f(\mathcal C)=X$. 
Identifying $S_i$ with $B$, let 
\[f_i:B\to Z_i\]
be the resulting map,
and let $F_i$ be a general fibre of $f_i$ which has codimension $m_i$.
Note that
\eqspl{fibre}{H^{m_i}.S_i\sim \deg(Z_i)F_i.}

Now, we have
\eqspl{}{
H^{m_i+1}S_i\pi^*(L)^{k-m_i-2}=0,
}
while, by surjectivity of $f$,
\eqspl{}{
H^{m_i+2}\pi^*(L)^{k-m_i-2}>0.
}
Therefore the Hodge index theorem implies that
\eqspl{hodge}{
H^{m_i}S_i^2\pi^*(L)^{k-m_i-2}<0.
}
Now as $S_i$ is a section, we have
\[\O_{S_i}(S_i)=2Q_i.\]
In view of \eqref{fibre}, \eqref{hodge} means
\eqspl{}{\deg(Z_i)F_i.c_1(P_i)L^{k-m_i-2}>0,
}
so we may assume
\eqspl{pos}{F_i.c_1(P_i).L^{k-m_i-2}>0, i=2,3.}
\par
 Now, since the sections $S_i$ are pairwise distinct,
the natural map $P_i\to Q_j$ must be nontrivial, hence injective, for all $i\neq j$, hence
 $c_1(P_2)$ has negative degree
 on 
a general curve section of 
$F_3$. Thus
\eqspl{nonpos}{
c_1(P_2).F_3.L^{k-m_i-2}< 0.
}
But this obviously contradicts \eqref{1=2}.
\end{proof}
\begin{rem}
	The situation of Theorem \ref{qdisjoint} is not a priori
	amenable to the usual bend-and-break because there is not necessarily
	a curve $f(C_b)$, much less a 1-parameter family of such, through
	given points $y_1\in Y_1, y_2\in Y_2$.
	\end{rem}
\begin{rem}
	The last part of the proof above can be shortened somewhat using
	the following (sketched) argument; we have presented the longer proof
	in the hope it may generalize better. By
	restricting to a 1-parameter subfamily going through a fixed, general point of
	$Y_2$, we may assume that $f$ contracts $S_3$ to a point
	while $B$ is 1-dimensional. Then the disjointness condition
	\ref{disjoint} implies that $S_2$ and $S_3$ are numerically
	equivalent. This, together with the fact that $S_2$ and $S_3$
	are distinct and $S_3$ is contracted, easily yields a contradiction.
	\end{rem}
The hypotheses of Theorem \ref{qdisjoint} afford tweaking in various
ways, for example the following.
\begin{thm}Under Notation and Conventions as in the Introduction,
	assume given subvarieties $Y_1, Y_2, Y_3\subset X$ meeting a general
	$f(C_b)$ such that\par 
	(i) each $Y_i$ has has codimension 3 or more;\par
	(ii) $\dim(Y_2\cap Y_3) +\dim (Y_1)\leq n-3$;\par
	(iii)  $Y_1\cap Y_2\cap Y_3=\emptyset$;\par
	(iv) the subfamily of $B$ consisting of curves $f(C_b)$ 
	that are contained in $Y_2\cap Y_3$
	is of codimension $>2$.
	\nl\noindent
	Then there is a reducible fibre $C_b$.
	\end{thm}
\begin{proof}
	We may assume each $Y_i$ corresponds to a section $S_i$ of $\cC/B$,
	which in turn corresponds to an exact sequence
	\[\exseq{P_i}{E}{Q_i}, i=1,2,3.\]
	If $S_2\cap S_3=\emptyset$ then $c_1(Q_2)=c_1(P_3)$ and we easily get 
	a contradiction as above because both $P_2$ and $P_3$ 
	inject into $Q_1$.\par
	Now suppose $S_2\cap S_3\neq\emptyset$. If $\pi(S_2\cap S_3)$
	has dimension $<n-2$, it yields an $(n-3)$- dimensional
	family entirely contained in $Y_2\cap Y_3$, against our hypotheses.
	Hence $S_2\cap S_3$ projects
	to an $(n-2)$-dimensional subfamily $B'\subset B$ and the restricted family
	$\cC'/B'$ has disjoint sections corresponding to $S_1$ and $S_2$
	which get contracted to $Y_1$ and $Y_2\cap Y_3$ respectively.
	By Assumption (ii) this family has $\infty^1$ members through
	a pair of fixed points on $Y_1$ and $Y_2\cap Y_3$, so standard
	bend-and-break applies.
	
	\end{proof}
Next we give a bend-and-break type result for overfilling families.
\begin{thm}\label{overfilling}
	Under Notations and Conventions as in the Introduction, assume further\par
	(i) there is a subvariety $Y\subset X$ 
	of codimension 2 or more such that a general $f(C_b)$ meets $Y$ in 2
	or more points;\par
	(ii)\ \ $\dim(B)\geq n$.\nl\noindent
	 Then there is a reducible fibre $C_b$.
	\end{thm}
\begin{proof}
	We begin as in the proof of Theorem \ref{qdisjoint},
	arguing for contradiction and base-changing so that 
$\cC=\P(E)$ with $c_1(E)=0$ numerically,
and so that we have 2 sections $S_1, S_2$
collapsing to $Y$.
Let $P_i\subset E$ be the line subbundle corresponding to $S_i$ as before
and let $F_0$ be a component of a general fibre of $f$ over $X$. 
Note $\dim(F_0)=\dim(B)+1-n\geq 1$; replacing $B$ by a suitable subfamily we may assume 
$\dim(F_0)=1$.\par
We next aim to show that the subfamily of curves going through
a fixed point of $X$ is disjoint from the subfamily 
where the sections $S_1$
and $S_2$ meet (i.e., informally speaking, where the curves are 'tangent'
to $Y$).
To this end, I claim now that $\pi^*(P_i).F_0=0, i=1,2$.
To see this let $F_i$ be a component of a general fibre of $f|_{S_i}$. 
Thus $\dim(F_i)\geq 2$.
Then using Hodge index as
above we see that for any ample line bundle $A$ on $B$,
\[\pi^*(A)^{\dim(F_i)} \pi^*(P_i).F_i>0.\]
Since $F_i$ lies on $S_i$ which projects
isomorphically to $B$, 
this implies that $\pi^*(P_i)|_{F_i}$ is $\Q$-effective. 
Since $F_i$ is filled up by curves algebraically
equivalent to $F_0$, it follows that
\[\pi^*(P_i).F_0\geq 0, i=1,2.\] 
Now as $S_1$ and $S_2$ are distinct, 
the composite of the injection $P_1\to E$
and the projection $E\to Q_2$ yields an injection
$P_1\to Q_2$ and thus  $Q_2-P_1$ is effective. Hence, as $F_0$
is a general fibre, we have
\[0\leq \pi^*(P_1).F_0\leq \pi^*(Q_2).F_0\leq 0.\]
Thus $P_i.F_0=0, i=1,2$, as claimed. Since $H^n$ is a positive multiple
of $F_0$, we have
\eqspl{}{
\pi^*(P_i).H^n=0, i=1,2.
}
Now recall the injection of invertible sheaves
\[P_1\to Q_2. \]
Its zero locus, which is numerically $Q_1+Q_2=-(P_1+P_2)$,
is just the locus of points in the base $B$ over which the sections
$S_1=\P(Q_1)$ and $S_2=\P(Q_2)$ intersect, 
i.e. \[Q_1+Q_2\equiv_{\mathrm{num}}\pi_*(S_1\cap S_2).\] Since we know
\[H^n.(\pi^*(Q_i+Q_j))=0,\]
we conclude that \[f(\pi\inv\pi(S_1\cap S_2)\subsetneq X.\] 
This means exactly that the locus of curves going through
a general point of $X$ is disjoint from the locus where $S_1$
and $S_2$ intersect.\par 
Now we can easily conclude.
Let $x\in X$ be general, and
let $B_1\to B$ be a component
of the  the normalization of $\pi(f\inv(x))$, which we may assume
is a smooth curve. Let $\cC_1/B_1$ be the pullback $\P^1$ bundle. 
Then $\cC_1$ is endowed
with 3 pairwise disjoint sections, namely $S'_1, S'_2$ corresponding to $S_1, S_2$
(disjoint because 
$x\not\in f(\pi\inv\pi(S_1+S_2))$), 
plus a section $T$ contracting to $x$.
But this is evidently impossible: writing the corresponding rank-2 bundle on $B_1$
as $A_1\oplus A_2$ corresponding to $S'_1, S'_2$, the subbundle corresponding
to $T$ is isomorphic on the one hand to $A_1$, on the other hand to $A_2$,
hence $A_1\simeq A_2$ and $\cC_1$ is a product bundle, which has no contractible sections.
Contradiction.\par
\end{proof}
\begin{example}
See \cite{hypersurf} for context and motivation. Let $Y$ be a 
smooth codimension-c
subvariety of $\P^n, c\geq 2$. Any component $B$ of the family of rational
curves of degree $e$
 meeting $Y$ in $a$ points is at least $(e+1)(n+1)-4-a(c-1)$-dimensional.
Assume $2\leq a\leq e$ and that a general curve in $B$ is smooth. 
For a general curve $C$ in the family, the normal bundle to
$C$ in $\P^n$ is a quotient of a sum of copies of $\O(e)$, hence it is
a direct sum of line bundles of degrees $\geq e$. Therefore since $a\leq e$,
the secant bundle $N^s_C$ is semipositive, hence the family is filling.
Hence it follows from Theorem \ref{overfilling}
that $B$ will parametrize some reducible curves.\par
It is shown in \cite{hypersurf} that when $Y$ is a $(d-1, d)$
complete intersection, there is only one component $B$ as above,
i.e. the family is irreducible,
which also implies the existence of reducibles in this case.
Another irreducibility result is given in the next section.
	\end{example}
\section{irreducibility}
For integers $e, a, n$ and a smooth subvariety $Y\subset \P^n$,
we denote by $M_e(a, Y)$ the projective
scheme parametrizing triples $(f, C, \a)$ where
$(f, C)$ is in the  Kontsevich space of stable maps
$f:C\to\P^n$ with $C$ of genus 0 and no marked points, 
and $\a$ is a length-$a$ subscheme $\a\subset f\inv(Y)$
(see \cite{ful-pan}). 
Our purpose is to prove
\begin{thm}
	Assume $Y$ is a general complete intersection of type $(c, d)$
	with $c+d\leq n, n\geq 3$. Then for all
	$a\leq e$, $M_e(a, Y)$ is irreducible of dimension
	$(e+1)(n+1)-4-a$ and for its general point, the following holds:
	the source $C$ is
$\P^1$, the map $f$ has degree 1, and $f\inv(Y)=\a$.
	Moreover if $a\leq e-1$ or $c+d\leq n-3$, $f$ is an embedding.
	\end{thm}
\begin{proof}
	The idea is to degenerate $Y$ to
	\[Y_0=(H_1\cup...\cup H_c)\cap(H_{c+1}\cup...\cup H_{c+d})\]
where $H_0,...,H_n$ form a basis for 
the hyperplanes in $\P^n$.
We will prove first that the assertions of the Theorem,
except for the irreducibility, which is false, 
hold for $Y_0$ in place of $Y$.\par

Write
\eqspl{subscheme}{\a=\sum\limits_p \a_p}
where $p\in C$ are distinct and
$\a_p$ is supported at $p$ and has length $a_p$ with
$\sum a_p=a$.
We begin by analyzing the case where $C$ is irreducible, i.e.
$C=\P^1$.  In that case $f$ amounts to an $(n+1)$ tuple of $e$-forms:
\[f=[\phi_0,...,\phi_n], \phi_i\in H^0(\O_C(e))\]
defined up to a constant factor and up to reparametrization.
Here $\phi_i=f^*(H_i)$. Because any component of $M_e(a, Y_0)$
has codimension at most $a$ in the space of all maps, while it is
$e+1>a$ conditions for any $\phi_i$ to vanish, so we may assume
all $\phi_i\neq 0$, i.e. $f(C)$ is not contained in any coordinate
hyperplane $H_i$. Also, the conditions on $f$ to appear below
will involve $\phi_i$ only for $i>0$, 
so the general $(n+1)$-tuple satisfying them
will have
no common zero, making the corresponding rational map a morphism.\par

To each $p$ appearing in \eqref{subscheme} we associate index-sets
\[I(p)\subset [1,c], J(p)\subset[c+1, c+d]\] 
with \[f(p)\in (\bigcup \limits_{i\in I(p)}H_i)\cap
(\bigcup\limits_{j\in J(p)}H_j).\]
Then
\eqspl{order}{a_p\leq\min(|I(p)|, |J(p)|).}
The condition \eqref{order} means that
\eqspl{}{
\sum\limits_{i\in I(p)}\ord_p(\phi_i)\geq a_p,
\sum\limits_{i\in J(p)}\ord_p(\phi_i)\geq a_p, 	
}
which amounts to $2a_p$ conditions on $f$: namely, 
if $L_p$ denotes a linear form with zero set $p$, that
\[L_p^{a_p}|\prod\limits_{i\in I(p)}\phi_i,
L_p^{a_p}|\prod\limits_{i\in J(p)}\phi_i,\] 
and for different points $p$
these conditions are linearly independent. In fact these conditions
define a union of linear spaces each of which is of the form
\[\{(\phi.):\ord_p(\phi_i)\geq b_{p,i}, \forall i\in I(p)\cup J(p)\},\]
where the $(b_{p, .})$ is  a sequence of nonnegative integers satisfying
\[\sum\limits_{i\in I(p)}b_{p,i}=a_p, \sum\limits_{i\in J(p)}b_{p,i}=a_p.\]
Thus it is $2a$ conditions to map a given subscheme $\a$ to $Y_0$
and $2a-r, r=|\supp(\a)|$ conditions on $f$
to map some unspecified  subscheme of type $(a_p)$
(i.e. isomorphic to $\a$ as above) to $Y_0$.
Since $a\geq r$ with equality iff $\a$ is reduced, it follows that 
$f(C)$ is transverse to $Y$. Also, an easy dimension count shows that
$f$ cannot have degree $>1$ to its image. 
Moreover, via multiplication by $\a$,
$\O_C(e-a)\to \O_C(e)$, the linear system corresponding to $f$ contains
$n+1$ unrestricted sections of $\O_C(e-a)$, which is very ample
if $a<e$. Finally if $c+d\leq n-3$ the system contains 4 or more
unresticted sections of $\O(e)$, namely $\phi_0, \phi_{c_d+1},...\phi_n$,
so again it is very ample.
Thus, we have shown
the Theorem
holds for the part of $M_e(a, Y_0)$ corresponding to irreducible curves.
\par
Next we analyze the case where $C$ has nodes. Having a node is 
already 1 condition on $f$ so it suffices to proves that
having a length-$a$ subscheme map into $Y_0$ is at least $a$
further conditions.
The map $f$ may be viewed as a projection of a rational normal
tree in $\P^e$, (connected) union of rational normal curves in 
their respective spans. 
The foregoing analysis 
goes through unchanged for points $p$ that are smooth on $C$, so suppose
$p$ is a node, with local branch coordinates $x, y$. The structure
of $\a_p$ is well understood (see \cite{hilb}) and in any case $\a_p$
contains a subscheme $\a'_p=Z(x^\alpha, y^\beta)$ with $\alpha, \beta>0$
and $\alpha+\beta\geq a_p$. Analyzing as above, it is at least
$2(\alpha+\beta-1)$ conditions on $f$ to map $\a'_p$, hence $\a_p$,
into $Y_0$, and this is $>\a_p$ unless $\alpha=\beta=1$. In the latter
case, if $\a_p=\a'_p$ then $\a_p=Z(x,y)$ has length 1 while the number
of conditions is 2. Finally, assume $\a'_p\neq\a_p$.
This means $\a_p$ is a tangent vector, i.e.
a length-2 locally principal subscheme of the form
 $\a_p=Z(x+ty), t\neq 0$. Note we may assume $I(p), J(p)$
are singletons, or else $f$ must map a node of $C$ to a proper
substratum of $Y_0$ (of dimension $n-3$ or less), which is at least 3
conditions. Then
the condition on $f$ to map $\a_p$ into $Y_0$ is first
 that it must map $p$,  a node of $C$,
to the top stratum of $Y_0$ (2 conditions),
and then the (2-dimensional) 
Zariski tangent space to $f(C)$ at $f(p)$ must be non-transverse to
$Y_0$, which is 3 conditions in total. This completes the proof of the 
Theorem, except for the irreducibility, for $Y_0$ and hence for $Y$.\par
Note that the components of $M_e(a, Y_0)$ are
of the form $M_e(a., i., j.)$
  where for each $k$,
$ 1\leq i_k\leq c<j_k\leq d$ and $\sum a_k=a$, and the general curve
in $M_e(a., i.,j.)$ has  $a_k$ points on 
the top stratum that is open dense in $H_{i_k}
\cap H_{j_k}$ for each $k$. Thus $M_e(a, Y_0)$ is highly reducible.
Anyhow for such a curve $C$, the normal bundle $N$ (strictly speaking, the
normal bundle to the map $f$) is a quotient of a sum of line bundles
of degree $e$ hence is itself a sum of line bundles of degree
$e$ or more. Consequently, thanks to the condition
$a\leq e$, the 'secant bundle' $N^s$, which parametrizes
deformations preserving the incidence to $Y$ (cf. \cite{filling}), is semipositive.\par
Note that each $M_e(a., i., j.)$ contains curves of the form
$C'\cup_xL$, where $C'$ is general in $M_{e-1}(a'., i'., j'.)$
where $(a'., i'., j'.)$ is obtained from $(a., i.,j.)$ by replacing
a single $a_k$ by $a_k-1$ (and omitting $(a_k, i_k, j_k)$ if $a_k=1$),
and $L$ is a line joining a general point 
$x=f(p)$ on $f(C)$ with a general point of 
$H_{i_k}\cap H_{j_k}$. Indeed such a curve is clearly unobstructed as secant
thanks to the fact that $N^s_L(-x)$ and $N^s_{C'}(-p)$ are 
both sums of line bundles of degree $-1$ or more, 
hence have vanishing $H^1$.\par
With that said, the irreducibility of $M_e(a, Y)$ follows easily
by induction on $e$, 
the case $e=1$ being trivial thanks to $Y$ itself being irreducible
(this is where we use $n\geq 3$): 
let $B$ be an irreducible
component of $M_e(a, Y), a\leq e$, and consider its limit $B_0$ in
$M_e(a, Y_0)$, which is a sum of components $M_e(a., i., j.)$. 
Because each of these contains an unobstructed curve of type
$C'_x\cup L$ as above, $B$ contains a similar curve of the form
$C"_x\cup L$ with $C"\in M_{e-1}(a-1, Y)$ and 
since the latter family may be assumed irreducible, $B$ is unique
so $M_e(a, Y)$ is irreducible.
\end{proof}
\begin{rem}
	The low-degree hypothesis on $Y$ does not seem necessary
	for irreducibility; 
	on the other hand absent some upper bound on $a$ like
	$a\leq e$, $M_e(a, Y)$ may contain components parametrizing
	curves having a component contained in $Y$ so
	irreducibility may fail. Another obvious question is as to the
	Kodaira dimension of $M_e(a, Y)$: probably maximal for large $e, a$
	but it's not clear. 
	\end{rem}

\vfill\eject
\bibliographystyle{amsplain}
\bibliography{../mybib}
\end{document}